\documentclass[12pt]{article}
\usepackage{mathrsfs}
\usepackage{graphicx,amsmath,amsfonts,amssymb,amscd,amsthm}
\newtheoremstyle{kai}
{3pt} {3pt} {} {} {\bfseries} {.} {.5em} {}
\def\EquationsBySection{\def\theequation
{\thesection.\arabic{equation}}%
\@addtoreset{equation}{section}}

\newcommand\old[1]{}
\newcommand{\pend}{\hfill \thicklines \framebox(6.6,6.6)[l]{}}

\renewenvironment{proof}{\noindent {\it  Proof.} \rm}{\pend}
\newtheorem{theorem}{Theorem}[section]
\newtheorem{lemma}{Lemma}[section]

\newtheorem{remark}{Remark}[section]
\newtheorem{definition}{Definition}[section]
\newtheorem{example}{Example}[section]

\EquationsBySection \makeatother
\usepackage{indentfirst}

\begin{document}
\pagestyle{plain}
\title
{\bf Stabilization of Partial Differential Equations  by L\'{e}vy Noise}
\author{Jianhai Bao  and Chenggui Yuan\thanks{{\it E-mail address:}
C.Yuan@swansea.ac.uk.}\\
Department of Mathematics,\\
 Swansea University, Swansea SA2 8PP, UK
\\
}
\date{}
\maketitle
\begin{abstract}{\rm We focus in this paper on the stochastic stabilization
problems of PDEs by L\'{e}vy noise. Sufficient conditions under
which the perturbed systems decay exponentially with a
general rate function are provided and some examples are constructed
to demonstrate the applications of our theory.
 }\\

\end{abstract}

 \noindent

\section{Introduction}

%Stabilization, one of the most important problems in stability
%theory, of ODEs by noise sources in It\^o's sense or Stratonovich's
%sense has been well developed, e.g., \cite{m94} and a survey  paper
%\cite{m07}, and the theories have been applied to construct feedback
%stabilisers.

Recently, the investigation of stabilization of partial differential equations (PDEs) has
received much more attention. Kwieci\'{n}ska in \cite{k99} constructs
the first example of a class of PDEs being stabilized, in terms of
Lyapunov exponents, by noise, and in \cite{k02} provides a
sufficient condition for exponential stabilization of a
deterministic equation of evolution $dX(t)=AX(t)dt$ in a separable,
real Hilbert space; For nonlinear PDEs, Caraballo et al.
\cite{clm01} and Caraballo et al. \cite{cgj03} establish some
results on stabilization of (deterministic and stochastic) PDEs;
Caraballo \cite{c06} gives a brief review on some recent results on
the stabilization effect produced by noise in phenomena modelled by
PDEs. Here we would like to point out that stabilization of ordinary differential equations (ODEs) by noise sources in It\^o's sense  has been well developed by Mao \cite{m94} and a survey  paper
\cite{m07}, and the theories have been applied to construct feedback
stabilisers.

As we know, the general theory of stochastic stabilization for ODEs
and PDEs makes use of Wiener process as the source of noise.
Recently, Applebaum and Siakalli \cite{as10} take some first steps
in stochastic stabilization of dynamical systems in which the noise
is a more general L\'{e}vy process. However, to our knowledge there
are few works devoting themselves to the investigation of the
analogous problems in infinite dimension. In this paper, we shall
extend the results of Applebaum and Siakalli \cite{as10} in finite
dimension to infinite dimension. The generalization from finite
dimension to infinite dimension is not straightforward, since we
need to overcome the difficulty from  infinite dimensional analysis.
We focus on an evolution equation that is perturbed by L\'{e}vy
noise, i.e., a Wiener process and an independent Poisson random
measure, and we shall provide some sufficient conditions under which
the perturbed system decay exponentially with a general rate
function.

The organization of this paper is as follows: In Section 2 we
perturb a class of heat equations by L\'{e}vy  noise, give the
corresponding explicit solution and then reveal that the perturbed
systems become  pathwise exponentially stable; For the
preparation of stabilization problems regarding much more general
PDEs, one stability criterion of almost sure decay with a general rate
function for stochastic partial differential equations (SPDEs)
 with jumps is provided in Section 3, which is
also interesting in its own right; By using the theories established
in Section 3, in the last section we give some sufficient conditions
under which the perturbed systems driven by L\'{e}vy noise with
small jumps or large jumps decay exponentially  with
a general rate function. Moreover, results obtained in Theorem
\ref{criterion} and Theorem \ref{th3} are sharp, comparing with
Theorem \ref{th1} and Theorem \ref{th2}.

\section{An Example: Stabilization of Heat Equations by L\'{e}vy Noise}
For a bounded domain $\mathcal {O}\subset\mathbb{R}^n$ with
$C^{\infty}$ boundary $\partial\mathcal {O}$, let $H:=L^2(\mathcal
{O})$ denote the family of all real-valued square integrable
functions, equipped with the usual inner product $\langle
f,g\rangle_H:=\int_{\mathcal {O}}f(x)g(x)dx, f,g\in H$ and norm
$\|f\|_H:=\left(\int_{\mathcal {O}}f^2(x)dx\right)^{\frac{1}{2}},
f\in H$. Let $\bigtriangleup$ be the classical Laplace operator
$\sum_{i=1}^n\frac{\partial^2}{\partial x^2}$, say $A$, from the
Sobolev space $H^1_0(\mathcal {O})\cap H^2(\mathcal {O})$, denoted
by $\mathcal {D}(A)$, to $H$, where $H^m(\mathcal {O}),m=1,2$,
consist of functions of $H$ whose derivatives $D^{\alpha}u$, in the
sense of distributions, of order $|\alpha|\leq m$ are in $H$ and
$V:=H^1_0(\mathcal {O})$ is the subspace of elements of
$H^1(\mathcal {O})$ vanishing in some generalized sense on
$\partial\mathcal {O}$. Furthermore, it is well known that there
exists an orthonormal basis of $H$, $\{e_n\}_{n\geq1},
n=1,2,\cdots$, satisfying (see, e.g., \cite[p142]{pz92})
\begin{equation}\label{eq00}
e_n\in \mathcal {D}(A),\ \ \ \ -Ae_n=\lambda_ne_n.
\end{equation}
Thus, for any $f\in H$, we can write
\begin{equation*}
f=\sum\limits_{n=1}^{\infty}f_ne_n, \mbox{ where } f_n=\langle f,
e_n\rangle.
\end{equation*}

Consider heat equation in the form:
\begin{equation}\label{eq1}
\begin{cases}
\frac{\partial Y}{\partial t}&=\triangle Y+\alpha Y,
  x\in\mathcal {O}, \ t > 0,\\
Y(t,x)&=0,  x\in\partial\mathcal {O},\ t>0,\\
Y(0,x)&=u^0(x),  x\in\mathcal {O},
\end{cases}
\end{equation}
where $\alpha$ is an arbitrary constant and
$Y=Y(t,x),t>0,x\in\mathcal {O}$.

By the properties of the heat equation, we can solve \eqref{eq1} by
\begin{equation}\label{eq2}
Y(t,x)=\sum\limits_{n=1}^{\infty}\exp\{(-\lambda_n+\alpha)t\}u_n^0
e_n(x), t\geq0.
\end{equation}
\begin{definition}
{\rm The limit
\begin{equation*}
\lambda(u^0):=\limsup\limits_{t\rightarrow\infty}\frac{1}{t}\ln\|Y(t)\|_H
\end{equation*}
is called the sample Lyapunov exponent of Eq. \eqref{eq1}. }
\end{definition}

In, e.g., \cite{k99}, the precise Lyapunov exponent of Eq.
\eqref{eq1} was revealed, which is described as the following lemma.
\begin{lemma}\label{deterministic}
{\rm For initial condition $u^0\neq0$, let $n_0:=\inf\{n:u^0_n\neq
0\}$. Then the Lyapunov exponent of Eq. \eqref{eq1} exists and is
given by
\begin{equation*}
\lambda(u^0)=-\lambda_{n_0}+\alpha.
\end{equation*}
}
\end{lemma}

 Since $\lambda_{n_0}$ depends on the initial data $u^0,$
  we only know the asymptotic behavior of  the solution with initial data
  $u^0$. For any initial condition $u^0\neq0$, note from Eq. \eqref{eq2}
  that
\begin{equation*}
\begin{split}
\frac{1}{t}\ln\|Y(t)\|_H&=\frac{1}{t}\ln\left(\sum\limits_{n=1}^{\infty}|\exp\{(-\lambda_n+\alpha)t\}u_n^0|^2\right)^{\frac{1}{2}}\\
&\leq-\lambda_1+\alpha+\frac{1}{t}\ln\|u^0\|_H,
\end{split}
\end{equation*}
which implies
\begin{equation*}
\lambda(u^0)\leq-\lambda_1+\alpha.
\end{equation*}
Hence, if $\lambda_1>\alpha$ then Eq. \eqref{eq1} is exponentially
stable, while, for $\lambda_1\leq\alpha$, in general, Eq.
\eqref{eq1} might  not be stable.

Now we consider the following stochastically perturbed system
corresponding to Eq. \eqref{eq1}
\begin{equation}\label{eq3}
\begin{cases}
dX(t)&=(\triangle X(t)+\alpha X(t))dt+\beta X(t)dW(t)+\int_{|y|\leq
r}\gamma(y)X(t^-)\tilde{N}(dt,dy),
  x\in\mathcal {O}, \ t > 0,\\
X(t,x)&=0,  x\in\partial\mathcal {O},\ t>0,\\
X(0,x)&=u^0(x),  x\in\mathcal {O}.
\end{cases}
\end{equation}
Here $\alpha,\beta$ are arbitrary constants, $W$ is a real-valued
Wiener process w.r.t. the probability space $\{\Omega,{\mathcal
F},\{{\mathcal F}_{t}\}_{t\geq0}, \mathbb{P}\}$,
$\tilde{N}(dt,dz):=N(dt,dz)-\lambda(dz)dt$ associated with a Poisson
random measure $N:\mathcal
{B}(\mathbb{R}_+\times\mathbb{R})\times\Omega\rightarrow\mathbb{N}\cup\{0\}$
with the characteristic measure $\lambda$ on the measurable space
$(\mathbb{R},\mathcal {B}(\mathbb{R}))$, where
$\lambda(\mathbb{R})<\infty$,
$\gamma:\mathbb{R}\rightarrow\mathbb{R}$ with $\gamma>-1$ and $r>0$
is some positive constant.

\begin{lemma}\label{explicit solution}
{\rm For any $\gamma>-1$, Eq. \eqref{eq3} has a unique strong
solution
\begin{equation*}
\begin{split}
X(t,x)&=\exp\Big(\Big(-\frac{1}{2}\beta^2+\int_{|y|\leq
r}(\ln(1+\gamma(y))-\gamma(y))\lambda(dy)\Big)t+\beta
W(t)\\
&\quad\ \ \ \ \ \ \ \ \ \ \ \ \ +\int_0^t\int_{|y|\leq
r}\ln(1+\gamma(y))\tilde{N}(ds,dy)\Big)Y(t,x),
\end{split}
\end{equation*}
where $Y(t,x)$ is the solution of Eq. \eqref{eq1}.}

\end{lemma}

\begin{proof}
By, e.g., \cite[Theorem 3.1]{z09}, Eq. \eqref{eq3} has a unique
strong solution $X(t,x)$.  Set
\begin{equation}\label{eq4}
\bar{v}(t,x):=\sum\limits_{n=1}^{\infty}z_n(t)e_n(x).
\end{equation}
Here $z_n(t)$ satisfies the following stochastic equation
\begin{equation}\label{eq22}
dz_n(t)=(-\lambda_n+\alpha)z_n(t)dt+\beta z_n(t)dW(t)+\int_{|y|\leq
r}\gamma(y)z_n(t^-)\tilde{N}(dt,dy)
\end{equation}
with initial condition $z_n(0)=u^0_n$. By the It\^o formula, one deduce that
\begin{equation}\label{eq6}
\begin{split}
z_n(t)&=u^0_n\exp\Big(\Big(-\lambda_n+\alpha-\frac{1}{2}\beta^2+\int_{|y|\leq
r}(\ln(1+\gamma(y))-\gamma(y))\lambda(dy)\Big)t+\beta
W(t)\\
&\quad\ \ \ \ \ \ \ \ \ \ \ \ \ +\int_0^t\int_{|y|\leq
r}\ln(1+\gamma(y))\tilde{N}(ds,dy)\Big).
\end{split}
\end{equation}
It is easy to see that for all $t\ge 0$,
$\bar{v}(t)\in\mathcal{D}(A)$ almost surely, since $e_n
\in\mathcal{D}(A)$. On the other hand, noting that
\begin{equation*}
A\bar{v}(t,x)=-\sum\limits_{n=1}^{\infty}\lambda_nz_n(t)e_n(x)
\end{equation*}
and putting \eqref{eq22} into \eqref{eq4}, one gets
\begin{equation*}
\begin{split}
\bar{v}(t,x)&=u^0(x)+\sum\limits_{n=1}^{\infty}\Big(\int_0^t(-\lambda_n+\alpha)z_n(s)ds+\int_0^t\beta
z_n(s)dW(s)\\
&\quad+\int_0^t\int_{|y|\leq
r}\gamma(y)z_n(s^-)\tilde{N}(ds,dy)\Big)e_n(x)\\
&=u^0(x)+\int_0^t[A\bar{v}(s,x)+\alpha\bar{v}(s,x)]ds+\int_0^t\beta\bar{v}(s,x)dW(s)\\
&\quad+\int_0^t\int_{|y|\leq
r}\gamma(y)\bar{v}(s^-,x)\tilde{N}(ds,dy).
\end{split}
\end{equation*}
In the light of uniqueness of solution,  we can conclude
$X(t,x)=\bar{v}(t,x)$ and the required result follows by
substituting \eqref{eq6} into \eqref{eq4} and combining \eqref{eq2}.
\end{proof}

For later applications, let us cite a strong law of large numbers
for local martingales, e.g., Lipster \cite{l80}, as the following
lemma.

\begin{lemma}\label{large numbers}
{\rm Let $M(t),t\geq0$, be a local martingale vanishing at $t=0$ and
$\zeta(t)$ be a continuous adapted non-negative increasing process
such that
\begin{equation*}
\lim\limits_{t\rightarrow\infty}\zeta(t)=\infty \mbox{ and
}\int_0^\infty\frac{d\langle M\rangle(s)}{(1+\zeta(s))^2}<\infty
\mbox{ a.s.},
\end{equation*}
where $\langle M\rangle(t):=\langle M,M\rangle(t)$ is Meyer's angle
bracket process. Then
\begin{equation*}
\lim\limits_{t\rightarrow\infty}\frac{M(t)}{\zeta(t)}=0 \mbox{ a.s.
}
\end{equation*}
}
\end{lemma}

\begin{remark}\label{remark}
{\rm Let
\begin{equation*}
\Psi^2_{\mbox{loc}}:=\left\{\Psi(t,z): \Psi(t,z) \mbox{ is
 predictable and }
\Big|\int_0^t\int_{\mathbb{Y}}|\Psi(s,z)|^2\lambda(du)ds<\infty\right\}.
\end{equation*}
For $\Psi\in\Psi^2_{\mbox{loc}}$, set
\begin{equation*}
M(t):=\int_0^t\int_{\mathbb{Y}}\Psi(s,z)\tilde{N}(ds,du).
\end{equation*}
Then, $M(t)$ is a local Martingale, and, by, e.g., Kunita
\cite[Proposition 2.4]{k10},
\begin{equation*}
\langle
M\rangle(t)=\int_0^t\int_{\mathbb{Y}}|\Psi(s,z)|^2\lambda(du)ds
\mbox{ and } [M](t)=\int_0^t\int_{\mathbb{Y}}|\Psi(s,z)|^2N(ds,du),
\end{equation*}
where $ [M](t):=[M,M](t)$, square bracket process (or quadratic
variation process) of $M(t)$. }
\end{remark}

\begin{theorem}\label{th1}
{\rm  Assume that $\gamma>0$ and
\begin{equation}\label{eq7}
\int_{|y|\leq r}\gamma^2(y)\lambda(dy)<\infty.
\end{equation}
The solution $X(t,x)$ of the perturbed Eq. \eqref{eq3} has the
following properties:\\
 \noindent(i) For initial condition $u^0\not= 0$
\begin{equation}\label{eq5}
\lim\limits_{t\rightarrow\infty}\frac{1}{t}\ln\|X(t)\|_H=-\lambda_{n_0}+\alpha-\frac{1}{2}\beta^2+\int_{|y|\leq
r}(\ln(1+\gamma(y))-\gamma(y))\lambda(dy) \quad a.s.
\end{equation}
In particular, the solution of Eq. \eqref{eq3} with initial data
$u^0$  will converge exponentially to zero  with probability one if
and only if
\begin{equation*}
\lambda_{n_0}>\alpha-\frac{1}{2}\beta^2+\int_{|y|\leq
r}(\ln(1+\gamma(y))-\gamma(y))\lambda(dy).
\end{equation*}

\noindent(ii) For any initial condition $u^0\not= 0$,
\begin{equation}\label{eq13}
\limsup\limits_{t\rightarrow\infty}\frac{1}{t}\ln\|X(t)\|_H\leq-\lambda_1+\alpha-\frac{1}{2}\beta^2+\int_{|y|\leq
r}(\ln(1+\gamma(y))-\gamma(y))\lambda(dy)  \quad a.s.
\end{equation}
In particular, the trivial solution  of Eq. \eqref{eq3} is almost
surely exponentially stable if
\begin{equation*}
\lambda_1>\alpha-\frac{1}{2}\beta^2+\int_{|y|\leq
r}(\ln(1+\gamma(y))-\gamma(y))\lambda(dy).
\end{equation*}

}
\end{theorem}

\begin{proof}
By Lemma \ref{explicit solution} it follows that
\begin{equation*}
\begin{split}
\frac{1}{t}\ln\|X(t)\|_H&=\frac{1}{t}\ln\|Y(t)\|_H-\frac{1}{2}\beta^2+\int_{|y|\leq
r}(\ln(1+\gamma(y))-\gamma(y))\lambda(dy)\\
&\quad+\beta \frac{1}{t}W(t)+\frac{1}{t}\int_0^t\int_{|y|\leq
r}\ln(1+\gamma(y))\tilde{N}(ds,dy).
\end{split}
\end{equation*}
Let
\begin{equation*}
\tilde{M}(t):=\int_0^t\int_{|y|\leq
r}\ln(1+\gamma(y))\tilde{N}(ds,dy),
\end{equation*}
due to Remark \ref{remark}, in addition to $\gamma\geq0$, we have
\begin{equation}\label{eq32}
\begin{split}
\langle\tilde{M}\rangle(t)=\int_0^t\int_{|y|\leq
r}(\ln(1+\gamma(y)))^2\lambda(dy)ds,
\end{split}
\end{equation}
and for $t\geq0$
\begin{align*}
\int_0^t\frac{d\langle\tilde{M}\rangle(s)}{(1+s)^2}ds
&=\int_{|y|\leq r}(\ln(1+\gamma(y)))^2\lambda(dy)\int_0^t\frac{1}{(1+s)^2}ds\\
&\le \int_{|y|\leq
r}\gamma^2(y)\lambda(dy)\int_0^t\frac{1}{(1+s)^2}ds<\infty.
\end{align*}
Applying, together with \eqref{eq7}, Lemma \ref{large numbers}
yields
\begin{equation*}
\lim_{t\rightarrow\infty}\frac{1}{t}W(t)=0 \mbox{ a.s. and }
\lim_{t\rightarrow\infty}\frac{1}{t}\int_0^t\int_{|y|\leq
r}\ln(1+\gamma(y))\tilde{N}(ds,dy)=0 \mbox{ a.s. }
\end{equation*}
The conclusion then follows from Lemma \ref{deterministic}.
\end{proof}

By the fundamental inequality
\begin{equation*}
\ln(1+x)\leq x \mbox{ for } x>-1,
\end{equation*}
we can deduce that the Lyapunov exponents of Eq. \eqref{eq3} are
less or equal to that of the counterpart with $\gamma=0$. On the
other hand, the significant fact we here  want to reveal is that
L\'{e}vy noise can also be used to stabilize some (stochastic) PDE.
For $\gamma=0$ in Eq. \eqref{eq3}, if $\alpha>\lambda_1$ and
$\beta^2\leq2(\alpha-\lambda_1)$, we do not know whether the
corresponding trivial solution is stable or not. But, if the system
is further perturbed by compensated Poisson integral with small
jumps, say $\int_{|y|\leq r}\gamma(y)v(t^-)\tilde{N}(dt,dy)$, we can
deduce that the perturbed system becomes more pathwise exponentially
stable. In particular, for $\gamma(u)\equiv b>0$, by letting
$b\rightarrow\infty$
\begin{equation*}
\int_{|y|\leq
r}(\ln(1+\gamma(y))-\gamma(y))\lambda(dy)=(\ln(1+b)-b)\lambda(|y|\leq
r)\rightarrow-\infty.
\end{equation*}
Consequently, for arbitrary $\alpha,\beta$, we can choose $b$
sufficiently large such that the perturbed system becomes more
pathwise exponentially stable.

Next we further perturb system \eqref{eq1} by L\'{e}vy noise with
large jumps into the form
\begin{equation}\label{eq14}
\begin{cases}
dX(t)&=(\triangle X(t)+\alpha X(t))dt+\beta X(t)dW(t)+\int_{|y|\geq
r}\gamma(y)X(t^-)N(dt,dy),
  x\in\mathcal {O}, \ t > 0,\\
X(t,x)&=0,  x\in\partial\mathcal {O},\ t>0,\\
X(0,x)&=u^0(x),  x\in\mathcal {O}.
\end{cases}
\end{equation}
 By Lemma \ref{explicit
solution} we have for $\gamma>-1$
\begin{equation*}
X(t,x)=\exp\Big(-\frac{1}{2}\beta^2t+\beta
W(t)+\int_0^t\int_{|y|\geq r}\ln(1+\gamma(y))N(ds,dy)\Big)Y(t,x),
\end{equation*}
where $Y(t,x)$ is the solution to Eq. \eqref{eq1}

Carrying out a similar argument to that of  Theorem \ref{th1}, we
can deduce the following results.

\begin{theorem}\label{th2}
{\rm  Assume that $-1<\gamma<0$ and
\begin{equation}\label{eq24}
\int_{|y|\geq r}(\ln(1+\gamma(y)))^2\lambda(dy)<\infty.
\end{equation}
Then the solution $X(t,x)$ of the perturbed Eq. \eqref{eq14} has the
following properties:\\
 \noindent(i) For initial condition $u^0\not= 0$
\begin{equation}\label{eq15}
\lim\limits_{t\rightarrow\infty}\frac{1}{t}\ln\|X(t)\|_H=-\lambda_{n_0}+\alpha-\frac{1}{2}\beta^2+\int_{|y|\geq
r}\ln(1+\gamma(y))\lambda(dy) \quad a.s.
\end{equation}
In particular, the solution of Eq. \eqref{eq14} with initial data
$u^0$  will converge exponentially to zero  with probability one if
and only if
\begin{equation*}
\lambda_{n_0}>\alpha-\frac{1}{2}\beta^2+\int_{|y|\geq
r}\ln(1+\gamma(y))\lambda(dy).
\end{equation*}

\noindent(ii) For any initial condition $u^0\neq0$,
\begin{equation}\label{eq16}
\limsup\limits_{t\rightarrow\infty}\frac{1}{t}\ln\|X(t)\|_H\leq-\lambda_1+\alpha-\frac{1}{2}\beta^2+\int_{|y|\geq
r}\ln(1+\gamma(y))\lambda(dy)  \quad a.s.
\end{equation}
In particular, the trivial solution  of Eq. \eqref{eq14} is almost
surely exponentially stable if
\begin{equation*}
\lambda_1>\alpha-\frac{1}{2}\beta^2+\int_{|y|\geq
r}\ln(1+\gamma(y))\lambda(dy).
\end{equation*}
}
\end{theorem}

\begin{remark}
{\rm Condition \eqref{eq24}, which enables  Lemma \ref{large
numbers} to be available, is true provided that there exits constant
$-1<c<0$ such that $c<\gamma<0$. Let $\gamma\equiv c$ with $-1<c<0$.
By \eqref{eq16}, for arbitrary $\alpha,\beta$ we can conclude that
the perturbed system \eqref{eq14} is almost surely exponentially
stable provided that $c\downarrow-1$. In other words, L\'{e}vy noise
with large jumps may also be used to stabilize (stochastic) PDE. }

\end{remark}

\section{A Criterion on Pathwise Decay with General Rate
Functions}

In the previous section we deduce that heat equations can be
stabilized by L\'{e}vy noise with small jumps or large jumps. Before
discussing stabilization problems for much more general PDEs (or
SPDEs ), in this section we shall provide a criterion on pathwise
decay with general rate functions for SPDEs driven by L\'{e}vy
noise, which is also interesting in its own right.

Let $H$ be a separable Hilbert space with inner product
$\langle\cdot, \cdot\rangle_H$ and $H^*$ its dual. Let $V$ be a
separable Banach space, equipped with norm $\|\cdot\|$, such that
$V\subset H$ continuously and densely. Then for its dual space $V^*$
it follows that $H^*\subset V^*$ continuously and densely.
Identifying $H$ and $H^*$ via the Riesz isomorphism we have that
\begin{equation*}
V\subset H\equiv H^{*}\subset V^{*},
\end{equation*}
continuously and densely and if $\langle\cdot, \cdot\rangle$ denotes
the dualization between  $V^{*}$ and $V$ (i.e. $\langle z,
v\rangle:=z(v)$ for $z\in V^*, v\in V $), it follows that
\begin{equation*}
\langle z, v\rangle=\langle z, v\rangle_H \mbox{ for } z\in H, v\in
V.
\end{equation*}
Let $(\Omega, \mathcal {F}, \mathbb{P})$ be a complete probability
space on which an increasing and right continuous family
${\{\mathcal {F}_t}\}_{t\geq0}$ of sub-$\sigma$-algebra of $\mathcal
{F}$ is defined. Let $(K,\langle\cdot,\cdot\rangle_K,\|\cdot\|_K)$
be a separable Hilbert space and $W(t), t\geq0$, a $K$-valued
cylindrical Wiener process defined formally by
\begin{equation*}
W(t)=\sum\limits_{k=1}^{\infty}\beta_k(t)e_k,\ \ t\geq0,
\end{equation*}
where $e_k, k\in\mathbb{N}$, is an orthonormal basis of $K$ and
$\beta_k(t),k\in\mathbb{N}$, is a sequence of real-valued standard
Brownian motions mutually independent on the probability space
$(\Omega,\mathcal {F},\mathbb{P})$. Let $\mathcal {L}_2:=\mathcal
{L}_2(K,H)$, the space of all Hilbert-Schmidt operators from $K$
into $H$. Then the space $\mathcal {L}_2$ is a separable Hilbert
space, equipped with the norm $ \|\Phi\|^2_{\mathcal
{L}_2}:=\mbox{trace}(\Phi\Phi^*) $ for $\Phi\in\mathcal {L}_2.$

Let $T>0$ and $\tilde{N}(dt,du):=N(dt,du)-dt\lambda(du)$ associated
with a Poisson random measure $N:\mathcal {B}(\mathbb{Z}\times
\mathbb{R}_+)\times\Omega\rightarrow\mathbb{N}\cup\{0\}$ with the
characteristic measure $\lambda$ on the measurable space
$(\mathbb{Z},\mathcal {B}(\mathbb{Z}))$. For each
$\mathbb{A}\in\mathcal {B}(\mathbb{Z})$, the Poisson random measure
$N((0,t]\times\mathbb{A})$ can be represented by a point process $p$
on $\mathbb{Z}$ with the domain $D_p$ as a countable subset of
$\mathbb{R}_+$, the collection of non-negative real numbers. That
is, $N(t,\mathbb{A})=\sum_{s\in D_p, s\leq t}I_{\mathbb{A}}(p(s))$.
For a bounded, measurable subset $\mathbb{Y}$ of $\mathcal
{B}(\mathbb{Z})$, denote by $M^{\lambda,2}([0,T]\times
\mathbb{Y}\times\Omega;H)$ the collection of all predictable
mappings $g(s,u,\omega): [0,T]\times
\mathbb{Y}\times\Omega\rightarrow H$ such that
\begin{equation*}
\mathbb{E}\int_0^T\int_{\mathbb{Y}}\|
g(t,u,\omega)\|_H^2dt\lambda(du)<\infty
\end{equation*}
and $D([0,T];H)$ the space of all c\`{a}dl\`{a}g paths from $[0,T]$
into $H$.  Let $I^p([0, T];V)$ denote the space of all $V$-valued
processes $x(t)$, which are $\mathcal {F}_t$-measurable from $[0,T]$
to $V$ and satisfy $ \mathbb{E}\int_0^T\|x(t)\|^pdt<\infty$.  We
further assume that $W$ and $N$ are independent throughout the
paper.

In this section we consider the following SPDE with jumps
\begin{equation}\label{eq17}
dX(t)=A(t,X(t))dt+g(t,X(t))dW(t)+\int_{\mathbb{Y}}h(t,X(t^-),u)\tilde{N}(dt,du)
\end{equation}
with initial condition $X(0)=x_0\in H$, where
$A(t,\cdot):V\rightarrow V^*, g(t,\cdot):V\rightarrow\mathcal {L}_2$
and $h(t,\cdot,\cdot):V\times\mathbb{Y}\rightarrow H$ are
progressively measurable. In what follows, we further assume that
$A(t,0)=0, g(t,0)=0$, $h(t,0,0)=0$, under which Eq. \eqref{eq17} has
the solution $X(t)\equiv0$ corresponding to the initial value
$x_0=0$. This solution is called the trivial solution or equilibrium
point. $X(t,x_0)$ denotes the solution of Eq. \eqref{eq17} starting
from $x_0$ at time $0$.

Since in this part we are mainly interested in stability analysis of
trivial solutions, we shall assume that, for each $T>0, x_0\in H$
and certain $p>1$, Eq. \eqref{eq17} has a unique strong solution in
$I^p([0, T];V)\cap D([0,T];H)$. For existence and uniqueness of
solutions of Eq. \eqref{eq17} under suitable conditions of
boundedness, coercivity, monotonicity and Lipschitz conditions on
the operators $A,f,g,h$, see, e.g., Ref. \cite{z09}.

\begin{definition}{\rm
 Suppose $\rho(t)\uparrow\infty$, as $t\rightarrow\infty$, is
some positive, non-decreasing, continuous function defined for
 $t>0$. The solution of Eq. \eqref{eq17} is
said to be almost surely stable with rate function $\rho(t)$ of
order $\gamma>0$ if and only if
\begin{equation*}
\begin{split}
\limsup\limits_{t\rightarrow\infty}\frac{\ln\|X(t,x_0)\|_H}{\ln
\rho(t)}\leq-\gamma, \mbox{ a.s. }
\end{split}
\end{equation*}
}
\end{definition}
It is obvious that such stability implies exponential stability,
polynomial stability and logarithm stability when $\rho(t)=e^t, 1+t,
\ln t$, with $ t>0$, respectively.

Next we prepare the exponential martingale inequality with jumps,
which plays a key role in our stability analysis.

\begin{lemma}\label{exponential martingale}
{\rm Let $\mathbb{Y}$ be a Borel subset of
$\mathbb{R}\backslash\{0\}$. Assume that $g:[0,\infty)\rightarrow
\mathbb{R}$ and $\Psi:[0,\infty)\times \mathbb{Y}\rightarrow
\mathbb{R}$ are both $\mathcal {F}_t$-adapted processes such that
for any $T>0$
\begin{equation*}
\int_0^T|g(t)|^2dt< \infty\mbox{ a.s. and }
\int_0^T\int_{\mathbb{Y}}|\Psi(t,y)|^2\lambda(dy)dt<\infty  \mbox{
a.s.}
\end{equation*}
Then for any positive constants $\alpha>0$
\begin{equation*}
\begin{split}
\mathbb{P}\Big\{\sup\limits_{0\leq t\leq
T}\Big[&\int_0^tg(s)dB(s)-\frac{1}{2}\int_0^t|g(s)|^2ds+\int_0^t\int_{\mathbb{Y}}\Psi(s,y)\tilde{N}(ds,dy)\\
&-\int_0^t\int_{\mathbb{Y}}[e^{\Psi(s,y)}-1-\Psi(s,y)]\lambda(dy)ds\Big]\geq\alpha\Big\}\leq
e^{-\alpha},
\end{split}
\end{equation*}
where $B$ is a real-valued Brownian motion, $N$ is a Poisson
counting measure with intensity $\lambda$, and  $B$ and $N$ are
independent.}
\end{lemma}

\begin{proof}
We here only sketch the argument since it is similar to that of
\cite[Theorem 5.2.9, p291]{a09}. Noting from, e.g., \cite[Corollary
5.2.2, p288]{a09},  that
\begin{equation*}
\begin{split}
&\exp\Big(\int_0^tg(s)dB(s)-\frac{1}{2}\int_0^t|g(s)|^2ds+\int_0^t\int_{\mathbb{Y}}\Psi(s,y)\tilde{N}(ds,dy)\\
&\ \ \ \ \ \ \ \ \ \
 -\int_0^t\int_{\mathbb{Y}}[e^{\Psi(s,y)}-1-\Psi(s,y)]\lambda(dy)ds\Big)
\end{split}
\end{equation*}
is a martingale, together with Doob's martingale inequality, we
complete the proof.
\end{proof}

Now let us state our main result of this section.

\begin{theorem}\label{asymptotic}
{\rm Assume that the solution of Eq. \eqref{eq17} satisfies that
$X(t,x_0)\neq0$ for all $t\geq0$ a.s. provided $x_0\neq0$ a.s. Let
$U\in C^{2,1}(H\times \mathbb{R}_+;\mathbb{R}_+)$ be a function such
that $U_x(t,x)\in V$ for any $x\in V,t\in\mathbb{R}_+$ and
$\varphi_1(t)\in\mathbb{R}$, $\varphi_2(t),\varphi_3(t)\geq0$ be
continuous functions. Assume further that  there exist constants
$p>0,m,\gamma, \tau\geq0$ and
$\theta\in \mathbb{R}$ such that for $(t,x)\in\mathbb{R}_+\times V$\\
\noindent (i) $\|x\|_H^p\rho^m(t)\leq U(t,x);$\\
\noindent (ii) $\mathcal {L}U(t,x)\leq\varphi_1(t)U(t,x)$, where
\begin{equation*}
\begin{split}
\mathcal {L}U(t,x)&:=U_t(t,x)+\langle A(t,x),U_x(t,x)\rangle+\frac{1}{2}\mbox{trace}(U_{xx}(t,x)g(t,x)g^*(t,x))\\
&\quad+\int_{\mathbb{Y}}[U(t,x+h(t,x,y))-U(t,x)-\langle U_x(t,x),
h(t,x,y)\rangle_H]\lambda(dy);
\end{split}
\end{equation*}
\noindent (iii) $Q U(t,x):=\|g^*(t,x)U_x(t,x)\|_K^2\geq\varphi_2(t)U^2(t,x)$;\\
\noindent (iv) For
$\Lambda(t,x,y):=\dfrac{U(t,x+h(t,x,y))}{U(t,x)}$,
\begin{equation*} \int_{\mathbb{Y}}\left[\ln
\Lambda(t,x,y)-\Lambda(t,x,y)+1\right]\lambda(dy)
:=J(t,x,h)\leq-\varphi_3(t);
\end{equation*}
 \noindent (v)
\begin{equation*}
\sup\limits_{t\geq0,x\in V}\int_{\mathbb{Y}}(\ln
\Lambda(t,x,y))^2\lambda(dy)<\infty \mbox{ and }
\sup\limits_{t\geq0,x\in
V}\int_{\mathbb{Y}}\Lambda(t,x,y)\lambda(dy)<\infty;
\end{equation*}
 \noindent (vi)
\begin{equation*}
\begin{split}
&\limsup\limits_{t\rightarrow\infty}\frac{\int_0^t\varphi_1(s)ds}{\ln\rho(
t)}\leq\theta,\ \ \ \
\liminf\limits_{t\rightarrow\infty}\frac{\int_0^t\varphi_2(s)ds}{\ln\rho(
t)}\geq\gamma,\\
&\liminf\limits_{t\rightarrow\infty}\frac{\int_0^t\varphi_3(s)ds}{\ln
\rho( t)}\geq\tau,\ \ \ \
\limsup\limits_{t\rightarrow\infty}\frac{t}{\ln \rho(
t)}=\mu<\infty.
\end{split}
\end{equation*}
Then the solution of Eq. \eqref{eq17} satisfies
\begin{equation*}
\limsup\limits_{t\rightarrow\infty}\frac{\ln\|X(t,x_0)\|_H}{\ln
\rho(t)}\leq-\frac{m+\tau+\gamma/2-\theta}{p},\ \ \ \
\mathbb{P}-\mbox{a.s.}
\end{equation*}
In particular, if $m+\tau+\gamma/2>\theta$, the solution of Eq.
\eqref{eq17} is almost surely stable with rate function
$\rho(t)>0$ of order $m+\tau+\gamma/2-\theta$.}
\end{theorem}

\begin{proof}
 For
simplicity, in what follows we write $X(t)$ instead of $X(t,x_0)$.
For $\delta\in(0,\frac{1}{2}]$, applying the It\^o formula to
$\delta\ln U(t,x),x\in V,$ w.r.t. $X(t),t\geq0$, strong solution of
Eq. \eqref{eq17},
\begin{equation*}
\begin{split}
\delta\ln U(t,X(t)) &=\delta\ln
U(0,x_0)+\delta\int_0^t\frac{\mathcal
{L}U(s,X(s))}{U(s,X(s))}ds-\frac{\delta}{2}\int_0^t\frac{QU(s,X(s))}{U^2(s,X(s))}ds\\
&\quad+\delta\int_0^t\int_{\mathbb{Y}}\Big[\ln\Lambda(s,X(s),y)-\Lambda(s,X(s),y)+1\Big]\lambda(dy)ds\\
&\quad+\delta\int_0^t\frac{\langle
U_x(s,X(s)),g(s,X(s))dW(s)\rangle_H}{U(s,X(s))}\\
&\quad+\delta\int_0^t\int_{\mathbb{Y}}\ln\Lambda(s,X(s^-),y)\tilde{N}(ds,dy).
\end{split}
\end{equation*}
By virtue of the exponential martingale inequality with jumps, Lemma
\ref{exponential martingale}, for any positive constants $T$ and
$\nu$
\begin{equation*}
\begin{split}
\mathbb{P}\Big\{\omega:\sup\limits_{0\leq t\leq
T}\Big[&\delta\int_0^t\frac{\langle
U_x(s,X(s)),g(s,X(s))dW(s)\rangle_H}{U(s,X(s))}-\frac{\delta^2}{2}\int_0^t\frac{QU(s,X(s))}{U^2(s,X(s))}ds\\
&+\delta\int_0^t\int_{\mathbb{Y}}\ln\Lambda(s,X(s^-),y)\tilde{N}(ds,dy)\\
&-\int_0^t\int_{\mathbb{Y}}\Big[\Lambda^{\delta}(s,X(s),y)-1-\delta\ln\Lambda(s,X(s),y)\Big]\lambda(dy)ds\Big]>\nu\Big\}\leq
e^{-\nu}.
\end{split}
\end{equation*}
Choose $T=n$ and $\nu=2\ln n$, where $n\in\mathbb{N}$, in the above
equation. Since $\sum_{n=1}^{\infty}\frac{1}{n^{2}}<\infty$, it
follows from the standard Borel-Cantelli lemma that there exists an
$\Omega_0\subseteq\Omega$ with $\mathbb{P}(\Omega_0)=1$ such that
for any $\omega\in\Omega_0$ we can find an integer $n_0(\omega)>0$
such that
\begin{equation*}
\begin{split}
&\quad\delta\int_0^t\frac{\langle
U_x(s,X(s)),g(s,X(s))dW(s)\rangle_H}{U(s,X(s))}+\delta\int_0^t\int_{\mathbb{Y}}\ln\Lambda(s,X(s^-),y)\tilde{N}(ds,dy)\\
&\leq2\ln
n+\frac{\delta^2}{2}\int_0^t\frac{QU(s,X(s))}{U^2(s,X(s))}ds+\int_0^t\int_{\mathbb{Y}}\Big[\Lambda^{\delta}(s,X(s),y)-1-\delta\ln\Lambda(s,X(s),y)\Big]\lambda(dy)ds
\end{split}
\end{equation*}
whenever $0\leq t\leq n$ and $ n\geq n_0(\omega)$. Hence, for any
$\omega\in\Omega_0$ and $0\leq t\leq n$, we have
\begin{equation*}
\begin{split}
\delta\ln U(t,X(t)) &\leq\delta\ln U(0,x_0)+2\ln
n\\
&\quad+\delta\int_0^t\frac{\mathcal
{L}U(s,X(s))}{U(s,X(s))}ds-\frac{\delta(1-\delta)}{2}\int_0^t\frac{QU(s,X(s))}{U^2(s,X(s))}ds\\
&\quad+\delta\int_0^t\int_{\mathbb{Y}}\Big[\ln\Lambda(s,X(s),y)
-\Lambda(s,X(s),y)+1\Big]\lambda(dy)ds\\
&\quad+\int_0^t\int_{\mathbb{Y}}\Big[\Lambda^{\delta}(s,X(s),y)-1
-\delta\ln\Lambda(s,X(s),y)\Big]\lambda(dy)ds,
\end{split}
\end{equation*}
where $ n\geq n_0(\omega)$. In the light of a Taylor's series
expansion, for sufficiently small $\delta>0$,
\begin{equation*}
\Lambda^{\delta}(t,x,y)=1+\delta\ln\Lambda(t,x,y)+\frac{\delta^2}{2}(\ln\Lambda(t,x,y))^2\Lambda^{\xi}(t,x,y),
\end{equation*}
where $\xi$ lies between $0$ and $\delta$. In what follows we shall
show for $\delta\in[0,\frac{1}{2}]$
\begin{equation}\label{eq42}
\sup\limits_{t\geq0, x\in
V}\int_{\mathbb{Y}}(\ln\Lambda(t,x,y))^2\Lambda^{\xi}(t,x,y)\lambda(dy)=:\eta<\infty.
\end{equation}
Note that
\begin{equation*}
\begin{split}
\int_{\mathbb{Y}}(\ln\Lambda(t,x,y))^2\Lambda^{\xi}(t,x,y)\lambda(dy)&=\int_{0<\Lambda(t,x,y)<1}(\ln\Lambda(t,x,y))^2\Lambda^{\xi}(t,x,y)\lambda(dy)\\
&\quad+\int_{\Lambda(t,x,y)\geq1}(\ln\Lambda(t,x,y))^2\Lambda^{\xi}(t,x,y)\lambda(dy)\\
&=:I_1+I_2.
\end{split}
\end{equation*}
For $0<\Lambda(t,x,y)<1$ and $0\leq\xi\leq\delta\leq\frac{1}{2}$, we
have $\Lambda^{\xi}(t,x,y)\leq1$. Hence, by virtue of condition (v)
\begin{equation*}
I_1\leq\int_{0<\Lambda(t,x,y)\leq1}(\ln\Lambda(t,x,y))^2\lambda(dy)<\infty.
\end{equation*}
On the other hand, recalling the fundamental inequality
\begin{equation*}
\ln x\leq 4(x^{\frac{1}{4}}-1) \mbox{ for } x\geq1,
\end{equation*}
and observing $\Lambda^\xi(t,x,y)\leq\Lambda^{\frac{1}{2}}(t,x,y)$
for $\Lambda(t,x,y)\geq1$ and $0\leq\xi\leq\delta\leq\frac{1}{2}$,
we can also deduce from (v) that
\begin{equation*}
I_2\leq16\int_{\Lambda(t,x,y)\geq1}\Lambda(t,x,y)\lambda(dy)\leq16\int_{\mathbb{Y}}\Lambda(t,x,y)\lambda(dy)<\infty.
\end{equation*}
Consequently, the conclusion \eqref{eq42} must hold. Thus, by
conditions (ii),(iii) and (iv), for any $\omega\in\Omega_0$ and
$0\leq t\leq n$ with $ n\geq n_0(\omega)$
\begin{equation*}
\begin{split}
\delta\ln U(t,X(t)) &\leq\delta\ln U(0,x_0)+2\ln n
+\delta\int_0^t\varphi_1(s)ds\\
&\quad-\frac{\delta(1-\delta)}{2}\int_0^t\varphi_2(s)ds-\delta\int_0^t\varphi_3(s)ds
+\frac{\delta^2\eta t}{2}.
\end{split}
\end{equation*}
 Now, in particular, for $n-1\leq t\leq n$ and $ n\geq (n_0(\omega)\vee
n_1(\epsilon))+1$, by condition $(i)$
\begin{equation}\label{eq41}
\begin{split}
\frac{\ln\|X(t)\|_H}{\ln\rho(t)}
&\leq-\frac{m}{p}+\frac{1}{p\ln\rho(t)}\Big[\ln U(x_0,0)
+\frac{2}{\delta}\ln n
+\int_0^t\varphi_1(s)ds\\
&\quad-\frac{1-\delta}{2}\int_0^t\varphi_2(s)ds-\int_0^t\varphi_3(s)ds
+\frac{\delta\eta t}{2}\Big].
\end{split}
\end{equation}
Recalling from (vi) that $\limsup_{t\rightarrow\infty}\frac{t}{\ln
\rho( t)}=\mu<\infty$, for $n-1\leq t\leq n$ and $ n\geq
(n_0(\omega)\vee n_1(\epsilon))+1$, hence
$\limsup_{t\rightarrow\infty}\frac{\ln n}{\ln \rho(
t)}=\limsup_{t\rightarrow\infty}\left(\frac{t}{\ln \rho(
t)}\times\frac{\ln n}{t}\right)=0$. Letting $n\uparrow\infty$, in
addition to (vi), leads to
\begin{equation*}
\limsup\limits_{t\rightarrow\infty}\frac{\ln\|X(t)\|_H}{\ln\rho(t)}\leq-\frac{m+\tau+(1-\delta)\gamma/2-\theta-\delta\mu\eta/2}{p}.
\end{equation*}
The conclusion follows from the arbitrariness of $\delta$.
\end{proof}

\begin{remark}
{\rm By the elementary inequality
\begin{equation*}
\ln x\leq x-1 \mbox{ for } x\geq0,
\end{equation*}
it follows that
\begin{equation*}
\int_{\mathbb{Y}}\left[\ln
\Lambda(t,x,y)-\Lambda(t,x,y)+1\right]\lambda(dy)\leq0.
\end{equation*}
Hence, the condition (iv) in Theorem \ref{asymptotic} is reasonable.
}
\end{remark}

\section{Stabilization of PDEs by L\'{e}vy Noise}
Combining the stability criterion established in Section $3$, in
this part we shall discuss the stabilization problems for much more
general PDEs through L\'{e}vy noise.

Consider the evolution equation
\begin{equation}\label{eq19}
dX(t)=A(t,X(t))dt
\end{equation}
with initial condition $X(0)=x_0\in H$, where,  for
$t\in\mathbb{R}_+$, $A(t,\cdot):V\rightarrow V^*$ with $A(t,0)=0$.
The natural question is: if Eq. \eqref{eq19} is not stable, can we
stabilize it using  L\'{e}vy noise? In this section, we shall
provide a positive answer to this question. Let us perturb problem
\eqref{eq19} into the form
\begin{equation}\label{eq20}
dX(t)=A(t,X(t))dt+g(t,X(t))dW(t)+\int_{\mathbb{Y}}\gamma(t,y)X(t^-)\tilde{N}(dt,dy),
\end{equation}
where $g,W,N$ are defined as in Eq. \eqref{eq17}, and
$\gamma(t,\cdot):\mathbb{Y}\rightarrow \mathbb{R}_+$.

\begin{theorem}\label{criterion}
{\rm Assume that the solution of Eq. \eqref{eq20} satisfies that
$X(t,x_0)\neq0$ for all $t\geq0$ a.s. provided $x_0\neq0$ a.s. Let
$\phi_1(t)\in\mathbb{R}$, $\phi_2(t)\in\mathbb{R}_+$ be continuous
functions and assume further that there exist constants
$\theta_1\in\mathbb{R},\theta_2>0,\theta_3\geq0$ such that for $(t,x)\in\mathbb{R}_+\times V$\\
(I) $2\langle A(t,x),x\rangle+\|g(t,x)\|_{\mathcal
{L}_2}^2\leq\phi_1(t)\|x\|^2_H$;\\
(II) $\|g^*(t,x)x\|_K^2\geq\phi_2(t)\|x\|^4_H$;\\
(III)
\begin{equation*}
\begin{split}
\sup\limits_{t\geq0}\int_{\mathbb{Y}}\gamma^2(t,y)\lambda(dy)<\infty
\mbox{ and } \limsup\limits_{t\rightarrow\infty}\frac{t}{\ln \rho(
t)}<\infty;
\end{split}
\end{equation*}
(IV)
\begin{equation*}
\begin{split}
&\limsup\limits_{t\rightarrow\infty}\frac{\int_0^t\phi_1(s)ds}{\ln\rho(
t)}\leq\theta_1,\ \ \ \
\liminf\limits_{t\rightarrow\infty}\frac{\int_0^t\phi_2(s)ds}{\ln\rho(
t)}\geq\theta_2,\\
&\liminf\limits_{t\rightarrow\infty}\frac{\int_0^t\int_{\mathbb{Y}}[\gamma(s,y)-\ln(1+\gamma(s,y))]\lambda(dy)ds}{\ln
\rho( t)}\geq\theta_3
\end{split}
\end{equation*}
Then the perturbed system \eqref{eq20} has the property
\begin{equation*}
\lim\sup\limits_{t\rightarrow\infty}\frac{\ln\|X(t,x_0)\|_H}{\ln
\rho(t)}\leq-\left[\theta_2+\theta_3-\frac{\theta_1}{2}\right],\ \ \
\ \mathbb{P}-\mbox{a.s.}
\end{equation*}
In particular, if $\theta_2+\theta_3>\frac{\theta_1}{2}$, the
solution of Eq. \eqref{eq20} is stable with rate
function $\rho(t)>0$ of order
$\theta_2+\theta_3-\frac{\theta_1}{2}$. }

\end{theorem}

\begin{proof}
Let $U(t,x)=\rho^m(t)\|x\|_H^2,(t,x)\in\mathbb{R}_+\times V$. Then,
 by (I) and (II)
\begin{equation*}
\mathcal {L}U(t,x)\leq
\left[\frac{m\rho'(t)}{\rho(t)}+\phi_1(t)+\int_{\mathbb{Y}}\gamma^2(s,y)\lambda(dy)\right]U(t,x)
\end{equation*}
and
\begin{equation*}
QU(t,x)=4\phi_2(t)U^2(t,x), \Lambda(t,x,y)=(1+\gamma(t,y))^2,
\end{equation*}
where $\mathcal {L}U$ and $QU$ are defined in Theorem
\ref{asymptotic}. It is easy to see that
\begin{equation*}
\varphi_1(t)=\frac{m\rho'(t)}{\rho(t)}+\phi_1(t)+\int_{\mathbb{Y}}\gamma^2(t,y)\lambda(dy),
\varphi_2(t)=4\phi_2(t),
\end{equation*}
and
\begin{equation*}
\varphi_3(t)=\int_{\mathbb{Y}}[2\ln(1+\gamma(t,y))-2\gamma(t,y)-\gamma^2(t,y)]\lambda(dy).
\end{equation*}
Moreover, by condition (III), together with $\gamma>0$, the
assumption (v) in Theorem \ref{asymptotic} holds. In the sequel,
carrying out a similar argument to that of Theorem \ref{asymptotic},
we can complete the proof.
\end{proof}

\begin{example}\label{ex1}
{\rm Let us return to the perturbed system \eqref{eq3}. For $u\in V$
let $U(t,u)=\|u\|_H^2, A(t,u):=Au+\alpha u$ and $g(t,u):=\beta u$.
Compute
\begin{equation*}
2\langle A(t,u),u\rangle+\|g(t,u)\|_{\mathcal
{L}_2}^2\leq(-2\lambda_1+2\alpha+\beta^2)U(t,u)
\end{equation*}
and
\begin{equation*}
\|g^*(t,u)u\|_K^2=\beta^2U^2(t,u).
\end{equation*}
Hence, in Theorem \ref{criterion}
\begin{equation*}
\phi_1(t)=-2\lambda_1+2\alpha+\beta^2 \mbox{ and }
\phi_2(t)=\beta^2.
\end{equation*}
Moreover, for $\rho(t)=e^t,t>0$ it is easy to see that
\begin{equation*}
\theta_1=-2\lambda_1+2\alpha+\beta^2,\theta_2=\beta^2 \mbox{ and }
\theta_3=\int_{\mathbb{Y}}[\gamma(y)-\ln(1+\gamma(y))]\lambda(dy).
\end{equation*}
Thus, by Theorem \ref{criterion} we have
\begin{equation*}
\limsup\limits_{t\rightarrow\infty}\frac{\ln\|X(t,x_0)\|_H}{t}\leq-\left[\theta_2+\theta_3-\frac{\theta_1}{2}\right],\
\ \ \ \mathbb{P}-\mbox{a.s.}
\end{equation*}
In particular, if $\theta_2+\theta_3>\frac{\theta_1}{2}$, that is,
\begin{equation}\label{eq31}
\lambda_1>\alpha-\frac{1}{2}\beta^2+\int_{|y|\leq
r}(\ln(1+\gamma(y))-\gamma(y))\nu(dy),
\end{equation}
then the perturbed system \eqref{eq3} is almost surely exponentially
stable. }
\end{example}

\begin{remark}
{\rm  Compared with Theorem \ref{th1}, Example \ref{ex1} shows
results obtained in Theorem \ref{criterion} are sharp. }
\end{remark}

We further perturb problem \eqref{eq19} into the form
\begin{equation}\label{eq30}
dX(t)=A(t,X(t))dt+g(t,X(t))dW(t)+\int_{\mathbb{Z}\backslash\mathbb{Y}}\psi(t,y)X(t^-)N(dt,dy),
\end{equation}
where $g,W,N$ are defined as in Eq. \eqref{eq17}, and $-1<\psi<0$.
Observe that  Eq. \eqref{eq30} can also be rewritten as
\begin{equation*}
\begin{split}
dX(t)&=\left[A(t,X(t))+g(t,X(t))dW(t)+\int_{\mathbb{Z}\backslash\mathbb{Y}}\psi(t,y)X(t)\lambda(dy)\right]dt\\
&\quad\ \ \
+\int_{\mathbb{Z}\backslash\mathbb{Y}}\psi(t,y)X(t^-)\tilde{N}(dt,dy).
\end{split}
\end{equation*}
\begin{theorem}\label{th3}
{\rm  Assume that the solution of Eq. \eqref{eq30} satisfies that
$X(t,x_0)\neq0$ for all $t\geq0$ a.s. provided $x_0\neq0$ a.s. Let
$-1<\psi<0$ and assume that there exist continuous function
$\xi(t)\in\mathbb{R},\eta(t)\geq0$ and constants
$\alpha_1\in\mathbb{R},\alpha_2\geq0,\alpha_3<0$ such that for
$(t,x)\in\mathbb{R}_+\times V$\\
(1) $2\langle A(t,x),x\rangle+\|g(t,x)\|_{\mathcal
{L}_2}^2\leq\xi(t)\|x\|^2_H$;\\
(2) $\|g^*(t,x)x\|_K^2\geq\eta(t)\|x\|^4_H$;\\
(3)
\begin{equation*}
\begin{split}
\sup\limits_{t\geq0}\int_{\mathbb{Y}}(\ln(1+\gamma(t,y)))^2\lambda(dy)<\infty
\mbox{ and } \limsup\limits_{t\rightarrow\infty}\frac{t}{\ln \rho(
t)}<\infty;
\end{split}
\end{equation*}
(4)
\begin{equation*}
\begin{split}
&\limsup\limits_{t\rightarrow\infty}\frac{\int_0^t\xi(s)ds}{\ln\rho(
t)}\leq\alpha_1,\ \ \ \ \
\liminf\limits_{t\rightarrow\infty}\frac{\int_0^t\eta(s)ds}{\ln\rho(
t)}\geq\alpha_2\\
&\limsup\limits_{t\rightarrow\infty}\frac{\int_0^t\int_{\mathbb{Z}\backslash\mathbb{Y}}\ln(1+\psi(s,y))\lambda(dy)ds}{\ln\rho(
t)}\leq\alpha_3.
\end{split}
\end{equation*}
Then the solution of Eq. \eqref{eq30} has the property
\begin{equation*}
\limsup\limits_{t\rightarrow\infty}\frac{\ln\|X(t,x_0)\|_H}{\ln\rho(
t)}\leq\frac{\alpha_1}{2}-\alpha_2+\alpha_3,\ \ \ \
\mathbb{P}-\mbox{a.s.}
\end{equation*}
In particular, if $\frac{\alpha_1}{2}-\alpha_2+\alpha_3<0$, then the
perturbed system \eqref{eq30} is stable with rate
function $\rho(t)$ of order
$-\left(\frac{\alpha_1}{2}-\alpha_2+\alpha_3\right)$. }
\end{theorem}

\begin{proof}
We shall make use of Theorem \ref{asymptotic}. Let
$U(t,x)=\rho^m(t)\|x\|_H^2,(t,x)\in\mathbb{R}_+\times V$, using
conditions (1) and (2) we have
\begin{equation*}
\mathcal {L}U(t,x)\leq
\left(\frac{m\rho'(t)}{\rho(t)}+\xi(t)+2\int_{\mathbb{Z}\backslash\mathbb{Y}}\psi(t,y)\lambda(dy)+\int_{\mathbb{Z}\backslash\mathbb{Y}}\psi^2(t,y)\lambda(dy)\right)U(t,x),
\end{equation*}
and
\begin{equation*}
QU(t,x)=4\eta(t)U^2(t,x), \Lambda(t,x,y)=(1+\psi(t,y))^2.
\end{equation*}
Also, we have
\begin{equation*}
\varphi_1(t)=\frac{m\rho'(t)}{\rho(t)}+\xi(t)+2\int_{\mathbb{Z}\backslash\mathbb{Y}}\psi(t,y)\lambda(dy)+\int_{\mathbb{Z}\backslash\mathbb{Y}}\psi^2(t,y)\lambda(dy),
\end{equation*}
and
\begin{equation*}
\varphi_2(t)=4\eta(t), \mbox{ and }
\varphi_3(t)=\int_{\mathbb{Z}\backslash\mathbb{Y}}[2\ln(1+\psi(t,y))-2\psi(t,y)-\psi^2(t,y)]\lambda(dy).
\end{equation*}
Furthermore,  thanks to condition (3), in addition to $-1<\psi<0$,
the assumption (v) in Theorem \ref{asymptotic} also holds. Then,
using the argument of that of Theorem \ref{asymptotic}, the
proof is therefore complete.
\end{proof}

\begin{example}
{\rm Let us  re-examine problem \eqref{eq14}. Noting that
\begin{equation*}
\xi(t)=-2\lambda_1+2\alpha+\beta^2 \mbox{ and } \eta(t)=\beta^2,
\end{equation*}
we have
\begin{equation*}
\alpha_1=-2\lambda_1+2\alpha+\beta^2,\alpha_2=\beta^2,\alpha_3=\int_{|y|\geq
r}\ln(1+\gamma(y))\nu(dy),
\end{equation*}
and by Theorem \ref{th3}
\begin{equation*}
\limsup\limits_{t\rightarrow\infty}\frac{1}{t}\log(\|\mu(t)\|_H)\leq-\lambda_1+\alpha-\frac{1}{2}\beta^2+\int_{|y|\geq
r}\ln(1+\gamma(y))\nu(dy)  \quad a.s.
\end{equation*}
In particular, the trivial solution  of Eq. \eqref{eq14} is pathwise
exponentially stable if
\begin{equation*}
\lambda_1>\alpha-\frac{1}{2}\beta^2+\int_{|y|\geq
r}\ln(1+\gamma(y))\nu(dy).
\end{equation*}
}

\end{example}

From (III) in Theorem \ref{criterion} and (3) in Theorem \ref{th3},
note that Theorem \ref{criterion} and Theorem \ref{th3} impose
constraints on $\gamma(t,\cdot), \psi(t,\cdot)$ and $\rho(t)$, in
what follows, we shall develop another theorem which omit these
restrictions.

\begin{theorem}
{\rm Assume that the solution of Eq. \eqref{eq20} satisfies that
$X(t,x_0)\neq0$ for all $t\geq0$ a.s. provided $x_0\neq0$ a.s. Let
$\phi_1(t)\in\mathbb{R}$, $\phi_2(t)\in\mathbb{R}_+$ be continuous
functions and assume further that there exist constants
$\beta_1\in\mathbb{R},\beta_2>0,\beta_3\geq0$ such that for $(t,x)\in\mathbb{R}_+\times V$\\
(i) $2\langle A(t,x),x\rangle+\|g(t,x)\|_{\mathcal
{L}_2}^2\leq\phi_1(t)\|x\|^2_H$;\\
(ii) $\phi_3(t)\|x\|^4_H\geq\|g^*(t,x)x\|_K^2\geq\phi_2(t)\|x\|^4_H$;\\
(iii)
\begin{equation*}
\limsup\limits_{t\rightarrow\infty}\frac{\int_0^t\phi_1(s)ds}{\ln\rho(
t)}\leq\beta_1,\ \ \ \
\liminf\limits_{t\rightarrow\infty}\frac{\int_0^t\phi_2(s)ds}{\ln\rho(
t)}\geq\beta_2,\ \ \ \
\limsup\limits_{t\rightarrow\infty}\frac{\int_0^t\phi_3(s)ds}{\ln\rho(
t)}<\infty
\end{equation*}
(iv)
\begin{equation*}
\begin{split}
&\liminf\limits_{t\rightarrow\infty}\frac{\int_0^t\int_{\mathbb{Y}}(\ln(1+\gamma(s,y)))^2\lambda(dy)ds}{\ln
\rho( t)}>0,\ \
\limsup\limits_{t\rightarrow\infty}\frac{\int_0^t\int_{\mathbb{Y}}(\ln(1+\gamma(s,y)))^2\lambda(dy)ds}{\ln
\rho( t)}<\infty\\
&\liminf\limits_{t\rightarrow\infty}\frac{\int_0^t\int_{\mathbb{Y}}[\gamma(s,y)-\ln(1+\gamma(s,y))]\lambda(dy)ds}{\ln
\rho( t)}\geq\beta_3.
\end{split}
\end{equation*}
Then the perturbed system \eqref{eq20} has the property
\begin{equation*}
\limsup\limits_{t\rightarrow\infty}\frac{\ln\|X(t,x_0)\|_H}{\ln
\rho(t)}\leq-\left[\beta_1+\beta_2-\frac{\beta_3}{2}\right],\ \ \ \
\mathbb{P}-\mbox{a.s.}
\end{equation*}
In particular, if $\beta_1+\beta_2>\frac{\beta_3}{2}$, the solution
of Eq. \eqref{eq20} is stable with rate function
$\rho(t)>0$ of order $\beta_1+\beta_2-\frac{\beta_3}{2}$. }

\end{theorem}

\begin{proof}
Our proof is motivated by the work \cite{cgj03}.  Applying the It\^o
formula to $\ln(\rho^m(t)\|x\|_H^2), (t,x)\in\mathbb{R}_+\times V$,
w.r.t. $X(t)$, solution of Eq. \eqref{eq20}, we have
\begin{equation*}
\begin{split}
\ln(\rho^m(t)\|X(t)\|_H^2)&=\ln(\rho^m(0)\|x_0\|_H^2)+\int_0^t\frac{m\rho'(s)}{\rho(s)}ds\\
&\quad+\int_0^t\frac{1}{\|X(s)\|_H^2}[2\langle A(t,X(s),
X(s))\rangle+\|g(s,X(s))\|^2_{\mathcal
{L}_2}]ds\\
&\quad-2\int_0^t\frac{\|g^*(s,X(s))X(s)\|_K^2}{\|X(s)\|_H^4}ds\\
&\quad+2\int_0^t\int_{\mathbb{Y}}[\ln(1+\gamma(s,y))-\gamma(s,y)]\lambda(dy)ds\\
 &\quad+2\int_0^t\frac{1}{\|X(s)\|_H^2}\langle X(s),
g(t,X(s))dW(s)\rangle_H+2\int_0^t\int_{\mathbb{Y}}\ln(1+\gamma(s,y))\tilde{N}(ds,dy).
\end{split}
\end{equation*}
By condition (i) and (ii) it follows that
\begin{equation*}
\begin{split}
2\ln(\|X(t)\|_H)&\leq\ln(\rho^m(0)\|x_0\|_H^2)+\int_0^t\phi_1(s)ds-2\int_0^t\phi_2(s)ds\\
&\quad+2\int_0^t\int_{\mathbb{Y}}[\ln(1+\gamma(s,y))-\gamma(s,y)]\lambda(dy)ds\\
 &\quad+2\int_0^t\frac{1}{\|X(s)\|_H^2}\langle X(s),
g(t,X(s))dW(s)\rangle_H+2\int_0^t\int_{\mathbb{Y}}\ln(1+\gamma(s,y))\tilde{N}(ds,dy).
\end{split}
\end{equation*}
Setting
\begin{equation*}
M(t):=\int_0^t\frac{1}{\|X(s)\|_H^2}\langle X(s),
g(t,X(s))dW(s)\rangle_H \mbox{ and
}\tilde{M}(t):=\int_0^t\int_{\mathbb{Y}}\ln(1+\gamma(s,y))\tilde{N}(ds,dy),
\end{equation*}
we have
\begin{equation}\label{eq40}
\int_0^t\phi_2(s)ds\leq\langle
M,M\rangle(t)=\int_0^t\frac{\|g^*(t,X(s))X(s)\|_K^2}{\|X(s)\|_H^4}ds\leq\int_0^t\phi_3(s)ds
\end{equation}
and
\begin{equation*}
\langle
\tilde{M},\tilde{M}\rangle(t)=\int_0^t\int_{\mathbb{Y}}(\ln(1+\gamma(s,y)))^2\lambda(dy)ds.
\end{equation*}
By (iii) and (iv) it is easy to see that $\mathbb{P}$-a.s.
\begin{equation*}
\langle M,M\rangle(t)\rightarrow\infty \mbox{ and }\langle
\tilde{M},\tilde{M}\rangle(t)\rightarrow\infty \mbox{ as }
t\rightarrow\infty
\end{equation*}
and, together with Lemma \ref{large numbers},
\begin{equation}\label{eq41}
\frac{M(t)}{\langle M,M\rangle(t)}\rightarrow0 \mbox{ a.s. and
}\frac{\tilde{M}(t)}{\langle
\tilde{M},\tilde{M}\rangle(t)}\rightarrow0 \mbox{ as }
t\rightarrow\infty.
\end{equation}
Moreover, note from \eqref{eq40}, \eqref{eq41} and (iii) that
\begin{equation*}
\frac{M(t)}{\ln\rho(t)}=\frac{M(t)}{\langle
M,M\rangle(t)}\frac{\langle M,M\rangle(t)}{\ln\rho(t)}\rightarrow0
\mbox{ a.s. and }\frac{\tilde{M}(t)}{\ln\rho(t)}\rightarrow0 \mbox{
a.s. }
\end{equation*}
as $t\rightarrow\infty$. Then the desired assertion follows from
(iii) and (iv) immediately.
\end{proof}

\end{document}